\documentclass[11pt]{amsart}

\usepackage{
amsfonts,
amsmath,
latexsym,
amssymb,
enumerate,
mathrsfs,
centernot,
color,
}

\usepackage[matrix,arrow]{xy}

\newcommand{\labbel}[1]{\label{#1} [[{\bf #1}]]}  
\renewcommand{\labbel}{\label}

\newcommand{\sharr}{\text{\tiny{$\rightarrow$}}} 

\newcommand{\nup}{{\centernot\uparrow}}

\newtheorem{theorem}{Theorem}[section]
\newtheorem{lemma}[theorem]{Lemma}

\newtheorem{proposition}[theorem]{Proposition} 
 
\newtheorem{corollary}[theorem]{Corollary} 
 
\newtheorem{claim}[theorem]{Claim} 
\newtheorem*{claim*}{Claim}

\newtheorem*{theorem*}{Theorem}
\newtheorem*{proposition*}{Proposition}
\newtheorem*{corollary*}{Corollary}
\newtheorem*{lemma*}{Lemma}
\newtheorem*{scholion*}{Scholion}

\theoremstyle{definition}
\newtheorem{definition}[theorem]{Definition}

\theoremstyle{remark}
\newtheorem{remark}[theorem]{Remark}

\newtheorem*{remark*}{Remark}
\newtheorem*{remarks*}{Remarks}

\newtheorem*{observation*}{Observation}

\numberwithin{equation}{section}

\begin{document}

\title{Multi-argument specialization semilattices}

\author{Paolo Lipparini} 

\email{lipparin@axp.mat.uniroma2.it}

\urladdr{http://www.mat.uniroma2.it/~lipparin}

\address{Dipartimento di Matematica\\Viale della Multa Ricerca
 Scientifica  \\Universit\`a di Roma ``Tor Vergata'' 
\\I-00133 ROME ITALY}

\subjclass{Primary  06A15; Secondary 54A05;  06A12}

\keywords{multi-argument specialization semilattice; closure semilattice; closure space; 
universal extension}

\date{\today}

\thanks{Work performed under the auspices of G.N.S.A.G.A. 
Work partially supported by PRIN 2012 ``Logica, Modelli e Insiemi''.
The author acknowledges the MIUR Department Project awarded to the
Department of Mathematics, University of Rome Tor Vergata, CUP
E83C18000100006.}

\begin{abstract}
If $X$ is a closure space with closure $K$,
we consider the semilattice
$(\mathcal P(X), \cup)$
endowed with further relations
$ x \sqsubseteq y_1, y_2, \dots, y_n$
(a distinct $n+1$-ary  relation for each $n \geq 1$),
whose  interpretation 
is
$x \subseteq Ky_1 \cup Ky_2 \cup \dots \cup  Ky_n $.

We present axioms for such \emph{multi-argument specialization semilattices} 
and show that this list of axioms is complete for substructures,
namely, every model satisfying the axioms can be embedded into
some structure associated to some closure space as in the previous sentence.
We also provide a canonical embedding of 
a multi-argument specialization semilattice into 
(the reduct of) some closure semilattice. 
 \end{abstract}

\maketitle  

\section{Introduction} \labbel{intro} 

The notion of \emph{closure} 
is pervasive in mathematics,
both in the topological sense  and in the sense of 
\emph{hull}, \emph{generated by}.
The general notion of a closure space
which can be abstracted from the above two cases
has been dealt with or foreshadowed by 
such mathematicians as 
  Schr\"oder, Dedekind, Cantor, Riesz, Hausdorff, Moore,
\v Cech,  Kuratowski, Sierpi\'nski, Tarski, Birkhoff and
Ore, as listed in Ern\'e \cite{E}, with applications, 
among other, to ordered sets, lattice theory,
logic, algebra,  topology, computer science
and connections with category theory.
Due to the importance of the notion, it is interesting to study
variations and weakenings. 

In detail,   a \emph{closure space} is
a set $X$ together with
a unary operation $K$ on $\mathcal P(X)$
such that  $x \leq Kx$ ($K$ is \emph{extensive}),
 $KKx=Kx$  (\emph{idempotent}) and 
$x \leq y$ implies $Kx \leq Ky$ (\emph{isotone}),  
for all subsets $x, y$ of $X$. 
In particular, by Kuratowski characterization,
a topological space is a closure space satisfying 
$ K\emptyset= \emptyset $
and $K(x \cup y)=Kx \cup Ky$,
for all subsets $x,y$ of $X$.
Similar to the case of topological spaces, closure
spaces have an equivalent characterization
in terms of a family of closed subsets.  
See \cite{E} for further details. 
 
An asymmetry in the otherwise
clean correspondence \cite{MT} between topology and 
its algebraic counterpart 
 given by the Kuratowski axioms
is the fact that continuous functions 
among topological spaces do not
preserve closures. Indeed, if $\varphi$  is a function 
between topological spaces   and
$\varphi ^\sharr$ denotes the direct image function associated to $\varphi$,
then $\varphi$  
is continuous if and only if 
$ \varphi ^\sharr (Kz) \subseteq K \varphi ^\sharr(z)$,  and
this is not the same as
 $ \varphi ^\sharr (Kz) = K \varphi ^\sharr(z)$.
Henceforth in \cite{mtt,mttlib,mttna} we started the study of
\emph{specialization semilattices}, that is, 
join semilattices endowed with 
a further coarser preorder $ \sqsubseteq $ satisfying 
the condition
\begin{align}
\labbel{s3}    \tag{S3}
 &a \sqsubseteq b  \ \&\  a_1 \sqsubseteq b
 \Rightarrow 
 a \vee a_1 \sqsubseteq b. 
\end{align}
If $X$ is a topological space,
then $( \mathcal P(X), \cup, \sqsubseteq )$
is a specialization semilattice,
where $ \sqsubseteq $ 
is the binary relation on 
$\mathcal P(X)$
defined by
$ a \sqsubseteq b$
if  $a \subseteq Kb$,
for $a, b \subseteq X$,  
and where $K$ is topological closure.
In \cite{mtt} we showed that every 
specialization semilattice can be represented in the above way.
Moreover, in this situation 
continuous functions correspond exactly to 
homomorphisms of specialization semilattices (more details below). 

In passing, let us mention 
that there are  many more examples of specialization semilattices, which do 
not arise from topological notions.
Any semilattice homomorphism 
$\varphi$  
induces a specialization 
on the domain semilattice by setting
$a \sqsubseteq b$ if $\varphi(a) \leq \varphi (b)$.
Moreover, every specialization semilattice can be represented
in this way. Thus specialization semilattices are at the same time both
``substructures'' of topological spaces   and quotients
of semilattices.
As an  example, if we consider
the quotient of $\mathcal P(X)$ modulo the ideal of finite subsets
of $X$ and the corresponding quotient function, 
the corresponding ``specialization'' is inclusion modulo finite.  
Similar  quotient constructions appear 
in many disparate settings, in various fields 
and with a wide range of applications.
For example, such structures appear, e.~g.,  in theoretical studies 
related to measure theory,
algebraic logic and even theoretical physics.
See \cite{mtt,mttlib} for further examples and details.

Let us now return to the 
representation of a specialization semilattice 
as a substructure of 
$( \mathcal P(X), \cup, \sqsubseteq )$.
As hinted above, under this interpretation, 
homomorphisms of specialization semilattices
are indeed functorial, that is, if $X$ and $Y$ are 
topological spaces and $\varphi:X \to Y$,
then $\varphi$  is continuous if and only if 
the  image function $\varphi^\sharr$
is a homomorphism of the associated specialization semilattices.   
See \cite{mtt}  for details.

All the above arguments apply also in the more general situation 
of closure spaces, thus specialization semilattices
can be thought at the same time as subreducts of 
topological spaces and as subreducts of closure spaces.
Namely, specialization semilattices  do not distinguish between the two kinds of structures.
In order to distinguish the two cases it is appropriate to introduce a 
ternary relation $R$ interpreted by
\begin{equation} \labbel{mult}     
\text{  $R(a;b,c)$ \quad if  \quad  
$a  \subseteq  Kb \cup Kc$. }
 \end{equation}
Then topological spaces satisfy
$R(a;b,c) \Leftrightarrow R(a;b \vee c, b \vee c)  $,
an equivalence which is generally false in closure spaces.
As above, continuous functions correspond to 
homomorphisms which respect $R$, and similarly
for the $n+1$-ary relations we are going to
introduce.
See \cite[Remark 6.3]{mtt}
 for further comments.

For notational simplicity, we shall write 
$ a \sqsubseteq^2 b,c$ for  $R(a;b,c)$ and similarly 
for relations of larger arity.
We shall frequently omit the superscript
and  maintain it only in a few cases for clarity.
 Formally, we shall consider a language
with the join-semilattice symbol
$\vee$ and with $n+1$-ary relation symbols $R_n$, for $n \geq 1$.
 Then $R_n(a,b_1, \dots, b_n)$
shall be abbreviated as 
$a \sqsubseteq^n b_1, \dots, b_n$,
or simply $a \sqsubseteq b_1, \dots, b_n$ and its intended interpretation
is given by $a  \subseteq  Kb_1 \cup \dots \cup  Kb_n$ in some closure space.
Notice that $ \sqsubseteq ^1$ is the case of 
the ``specializations'' we have dealt with in the above discussions.

Here we present an axiomatization of such
multi-argument specialization semilattices, and show that
any model of the axioms can be embedded  
into a closure space, under the above interpretation.
Moreover, every multi-argument specialization semilattice 
has a canonical ``free'' extension into a principal specialization semilattice,
that is, a specialization semilattice in which closure always exists.
This is a result parallel to \cite{mttlib,mttna}.

\section{Preliminaries} \labbel{prel}

A \emph{closure operation} on a poset
is a unary operation $K$ which is extensive,
 idempotent and isotone.
A \emph{closure semilattice} is a join-semilattice endowed with a closure operation.
See \cite[Section 3]{E} for  a detailed study of closure posets
and semilattices,
with many applications.
Similar to the case of closure 
spaces and topologies, a closure operation on a poset induces
a \emph{specialization $ \sqsubseteq $} as follows:
$ x \sqsubseteq y$ if $x \leq Ky$.
Here $ \sqsubseteq $ is $ \sqsubseteq ^1$,
in the hierarchy of relations introduced in the
above section.  
Thus if $(P, \vee, K)$ is a closure semilattice,
then  $(P, \leq, \sqsubseteq )$ is a specialization semilattice. 
From the latter structure we can retrieve the original closure $K$: indeed,
$Kx$ is the $\leq$-largest element $y$ such that $ y \sqsubseteq x$.  
However, 
such a largest element does not necessarily exist
  in an arbitrary specialization semilattice: consider, for example
the specialization given by inclusion modulo finite.
See \cite{mtt,mttlib,mttna} for more details. 
If the specialization semilattice $\mathbf S$  is such that, for every $x \in S $,
there exists the  largest element $y \in S$ such that $ y \sqsubseteq x$,
then $\mathbf S$ is said to be \emph{principal}. 
Such an $y$  shall be denoted
by $Kx$, as well.

\begin{remark} \labbel{psc} 
(a) It can be shown
that there is a bijective correspondence between closure semilattices and
principal specialization semilattices. See \cite[Section 3.1]{E}, in particular,
\cite[Proposition 3.9]{E} for details.

(b) If $a$ and $b$ are elements of  some specialization semilattice 
and both $Ka$ and  $Kb$ exist, then
$Ka \leq Kb$ if and only if $ a \sqsubseteq b$.
See \cite[Remark 2.1(c)]{mttlib} for a proof.
\end{remark}

The notions of \emph{homomorphism} 
and \emph{embedding}  are always intended
in the classical model-theoretical sense \cite{H}.
In detail, a homomorphism
for the relation $ \sqsubseteq ^n$ 
is a function $\varphi$
such that $ a \sqsubseteq^n b_1, \dots, b_n$ 
implies $\varphi(a) \sqsubseteq^n  \varphi(b_1),  \dots,
\varphi(b_n)$.
An embedding is an injective homomorphism such that 
also the converse holds.
A homomorphism of closure semilattices 
is a semilattice homomorphism satisfying
$ \varphi (Ka)= K \varphi (a)$. 

Homomorphisms of specialization semilattices
do not necessarily preserve closures.
If $\varphi$ is a homomorphism between two \emph{principal}
specialization semilattices, then $\varphi$ is a \emph{$K$-homomorphism} 
if  $ \varphi (Ka)= K \varphi (a)$, for every $a$. Thus $K$-homomorphisms
are actually homomorphisms for the associated
closure semilattices.

Other unexplained notions and notation
are from \cite{mtt,mttlib,mttna}.

\section{Multi-argument specialization semilattices} \labbel{msec} 

\begin{definition} \labbel{masl}    
A join semilattice $\mathbf M$ with a further
family of relations $(R_n) _{n \geq 1} $ 
is a \emph{multi-argument specialization semilattice} if  
the following conditions hold
\begin{align} \labbel{m1} \tag{M1}    
& a \sqsubseteq a
\\ \labbel{m2}   \tag{M2}    
&a \sqsubseteq  b_1, b_2, \dots, b_n \  \& \ 
b_1 \sqsubseteq  c
\Rightarrow  a \sqsubseteq c, b_2, \dots, b_n  
\\ \labbel{m3}   \tag{M3}
& a \leq b  \  \& \   b \sqsubseteq c_1, \dots, c_m
\Rightarrow
a  \sqsubseteq c_1, \dots, c_m
\\ \labbel{m4}   \tag{M4}    
&a \sqsubseteq  b_1, \dots, b_n \Rightarrow  
a \sqsubseteq  b_{ \sigma 1}, \dots, b_{ \sigma n}, 
\\ &\nonumber \qquad\qquad\qquad
 \text{ for every permutation $\sigma$ of $\{ 1, 2, \dots, n\}$} 
\\ \labbel{m5}   \tag{M5}    
&a \sqsubseteq  b_1, \dots, b_n,  b_n, 
\Rightarrow  a \sqsubseteq   b_1, \dots, b_n
\\ \labbel{m6}   \tag{M6}    
&a \sqsubseteq  b_1, \dots, b_n 
\Rightarrow  a \sqsubseteq  b_1, \dots, b_n, b_{n+1} 
\\ \labbel{m7}   \tag{M7}    
&a \sqsubseteq  b_1, \dots, b_n \  \& \ 
a_1 \sqsubseteq  b_1, \dots, b_n
\Rightarrow  a \vee a_1 \sqsubseteq  b_1, \dots, b_n  
    \end{align} 
for every $a,a_1,b,b_1, \dots, c \in M$ and where, as mentioned in the introduction,
$ a  \sqsubseteq  b_1, \dots, b_n$ is a shorthand for 
$ R_n(a, b_1, \dots, b_n)$.
Thus, formally, a multi-argument specialization semilattice 
is a structure of the form $(\mathbf M, \vee, R_n) _{n \geq 1} $.
For simplicity, we shall write
$(\mathbf M, \vee, \sqsubseteq^n )  _{n \geq 1}$. 
\end{definition}

We first derive some consequences from \eqref{m1} - \eqref{m7}.

From \eqref{m1} and \eqref{m3} with 
$m=1$ and $b=c_1$ we get    
\begin{align} \labbel{m1+} \tag{M1+}    
&a \leq b \Rightarrow a \sqsubseteq b.
    \end{align}

By taking $n=1$ in \eqref{m2} we get
\begin{align} \labbel{m2-}    \tag{M2$-$}
&a \sqsubseteq  b \  \& \ 
b \sqsubseteq  c
\Rightarrow  a \sqsubseteq c, 
 \end{align}   
and then using \eqref{m1+} we get
\begin{align} \labbel{m2*}    \tag{M2*}
&a \sqsubseteq  b \  \& \ 
b \leq  c
\Rightarrow  a \sqsubseteq c, 
 \end{align}
   
By iterating \eqref{m2} and using \eqref{m4},
we get  
\begin{align} \labbel{m2+} \tag{M2+}    
&a \sqsubseteq  b_1, b_2, \dots, b_n \  \& \ 
b_1 \sqsubseteq  c_1\  \& \ \dots \  \& \ 
b_n \sqsubseteq  c_n
\Rightarrow  a \sqsubseteq c_1, c_2, \dots, c_n  
 \end{align}
 
From $c_i \sqsubseteq c_i$
given by \eqref{m1}  and
iterating \eqref{m4} and \eqref{m6}, we get  
$ c_i  \sqsubseteq  c_1, \dots, c_m$, for every 
$i \leq m$. 
Then, by iterating \eqref{m7},
we get  
\begin{equation}\labbel{pupb}    
  c_1 \vee  \dots \vee  c_m  \sqsubseteq  c_1, \dots, c_m.
 \end{equation} 
If $ a \leq c_1 \vee  \dots \vee  c_m$,
then from \eqref{pupb} and 
 \eqref{m3} with $b=c_1 \vee  \dots \vee  c_m$
  we get   $a \sqsubseteq  c_1, \dots, c_m$, henceforth 
\begin{align} \labbel{m8} \tag{M8}    
&a \leq c_1 \vee \dots \vee c_m \Rightarrow 
a \sqsubseteq  c_1, \dots, c_m.
 \end{align} 

Since $c_i  \sqsubseteq c_1 \vee \dots \vee c_m $, for every 
$i \leq m$, by  \eqref{m1+},  if   $a \sqsubseteq  c_1, \dots, c_m$,
then we can repeatedly apply \eqref{m2} and \eqref{m5}
(using \eqref{m4})
in order to get   $a \sqsubseteq   c_1 \vee \dots \vee c_m$.    
Thus
\begin{align} \labbel{m9}  \tag{M9}
a \sqsubseteq  c_1, \dots, c_m \Rightarrow
a \sqsubseteq   c_1 \vee \dots \vee c_m.
 \end{align} 

By \eqref{m4} - \eqref{m5}, the relation
$a \sqsubseteq b_1, \dots, b_n$ 
depends only on the set
 $B=\{ b_1, \dots, b_n\}$, rather  
than on the sequence $(b_1, \dots, b_n)$,
henceforth we could have written
$a \sqsubseteq \{ b_1, \dots, b_n\}$ 
or even $a \sqsubseteq B $ in place of 
$a \sqsubseteq b_1, \dots, b_n$.
 
By iterating \eqref{m6} we thus get 
\begin{align} \labbel{m6+} \tag{M6+}    
&a \sqsubseteq B  \ \& \  B \subseteq  B_1  \Rightarrow 
a \sqsubseteq  B_1.
 \end{align}

\begin{definition} \labbel{reg}
Recall the definition of a principal specialization semilattice 
from Section \ref{prel}. The notion of being principal
involves only the binary relation $ \sqsubseteq ^1$. 
On the other hand, the following definition of regularity
involves all the relations $ \sqsubseteq ^n$. 

A principal multi-argument specialization semilattice 
is \emph{regular} when  
$a \sqsubseteq  b_1, \dots, b_n$
holds  if and only if 
$a \leq  Kb_1 \vee \dots \vee Kb_n$.

Notice that, by  \eqref{m2+},
the ``if'' condition holds in every principal
multi-argument specialization semilattice.
 \end{definition}   

The following proposition is evident.
It asserts that the axioms \eqref{m1} - \eqref{m7} 
are true in the intended model and, 
moreover, the notion of homomorphism is preserved.

\begin{proposition} \labbel{vere} 
(1) Suppose that $\mathbf S$ is a closure semilattice and set 
$a \sqsubseteq  b_1, \dots, b_n$ if 
$a \leq K b_1 \vee \dots \vee K b_n$. Then
$S$ acquires the structure of 
a principal regular multi-argument specialization semilattice,
which we shall call the \emph{associated}, or
the  \emph{multi-argument specialization reduct} of $\mathbf S$.  

In particular, the above statements apply
when $\mathbf S$ is a closure space.

(2) If $\mathbf T$ is another closure semilattice and $\varphi: \mathbf S
\to \mathbf T$ is a homomorphism of closure semilattices, then 
$\varphi$  is a homomorphism between the    associated 
multi-argument specialization semilattices.
\end{proposition}

\section{Free principal extensions} \labbel{secimp} 

As in \cite{mttna}, the existence of the ``universal''
extensions and morphisms which  we are going to construct
 follows from abstract categorical arguments;
see, e.~g., \cite[Lemma 4.1]{mttlib}.
Needles to say, an explicit description of
such universal objects is sometimes very hard to find.

\begin{remark} \labbel{rough}
As in \cite{mttna}, we work on the ``free''
semilattice extension  $ \widetilde { \mathbf   M} $
 of $\mathbf M$ obtained by adding
a  set of new elements
$\{ \, Ka \mid a \in M \,\}$.
The required  conditions here are 
$a \leq Ka$ and 
\begin{equation*}
\text{$a  \leq  Kb_1 \vee \dots \vee Kb_n$,
 whenever $a \sqsubseteq_{_{ \mathbf   M}} b_1, \dots, b_n$}
\end{equation*} 
Compare Definition \ref{reg}.
We shall denote the ``old'' relations and operations 
 with the subscript of the parent structure; the
``new'' ones
will be unsubscripted.
We want to define $ \sqsubseteq $ 
 in $ \widetilde { \mathbf   M} $
in such a way that the new element $Ka$
is actually the closure of $a$ and, moreover,
we want the resulting multi-argument specialization semilattice
to be regular. 
If we want 
that the above conditions  hold,
 then  we necessarily have
  \begin{enumerate}[$\bullet$]
    \item  
$a \leq c \vee Kd_1 \vee \dots \vee Kd_k$, whenever 
there is an  elements $d \sqsubseteq_{_{ \mathbf   M}} d_1,
 \dots , d_k$ 
 in $\mathbf M$ such that 
$a \leq_{_{ \mathbf   M}} c \vee_{_{ \mathbf   M}}
 d$, and 
\item
$Kb \leq c \vee Kd_1 \vee \dots \vee Kd_k$, if
there is $j \leq k$ such that  
   $b \sqsubseteq_{_{ \mathbf   M}} b_j$.
  \end{enumerate}
Such considerations justify  clauses
(a) - (b) in Definition \ref{und2} below.
Compare Remark \ref{*} below.

The  construction hinted above adds
a new closure of $a$ in the extension, even when $a$ has already
a closure in $\mathbf M$. 
In a parallel situation in \cite[Section 4]{mttna}
we have constructed an extension which preserves 
a specified set of closures,
but we shall not pursue an analogue
construction here.  
\end{remark}   
 
If $M$
is any set,  
let $\mathbf M^{{<} \omega }$
be the semilattice 
of the finite subsets of $M$,
with the operation of union.

\begin{definition} \labbel{und2}   
Suppose that 
$\mathbf M=
(M, \vee_{_{ \mathbf   M}}, \sqsubseteq ^n_{_{ \mathbf   M}}) _{n \geq 1}$
 is a multi-argument specialization semilattice.
On the product
$ M \times  M^{{<} \omega }$ 
define the following relations
  \begin{enumerate}[(a)] 
\item
 $(a, \{b_1, b_2, \dots, b_h\}) \precsim 
(c,\{d_1, d_2, \dots, d_k\})$ if
   \begin{enumerate}[({a}1)]  
  \item 
there is an  elements $d \sqsubseteq_{_{ \mathbf   M}} d_1,
 \dots , d_k$ 
 in $M$ such  that
 $a \leq_{_{ \mathbf   M}} c
 \vee_{_{ \mathbf   M}} d$
(if $k=0$, that is, if $\{d_1, d_2, \dots, d_k\}$ is empty,
the clause simply reads  $a \leq_{_{ \mathbf   M}} c $), and
   \item 
for every $i \leq h$,   
there is $j \leq k$ such that 
 $b_i \sqsubseteq_{_{ \mathbf   M}} d_j$.
 \end{enumerate}
\item
 $(a, \{b_1, b_2, \dots, b_h\}) \sim 
(c,\{d_1, d_2, \dots, d_k\})$ if both

 $(a, \{b_1, b_2, \dots, b_h\}) \precsim 
(c,\{d_1, d_2, \dots, d_k\})$ and

 $(c,\{d_1, d_2, \dots, d_k\}) \precsim 
(a, \{b_1, b_2, \dots, b_h\}) $.
  \end{enumerate}

In Lemma \ref{corre2} we shall prove 
that $\sim$ is an equivalence relation.   

Let 
$ \widetilde {  M} $
be the quotient of  
$ M \times M^{{<} \omega }$ 
under the equivalence relation $\sim$.

Define $K: \widetilde {   M} \to \widetilde {   M}$
by
\begin{equation*}   
K[a, \{b_1, \dots, b_h\}] =
 [a, \{a \vee_{_{ \mathbf   M}} b_1 \vee_{_{ \mathbf   M}}
 \dots \vee_{_{ \mathbf   M}} b_h  \} ],
  \end{equation*}    
where $[a, \{b_1,  \dots, b_h\}]$ 
denotes the $ \sim$-class of the pair $(a, \{b_1,  \dots, b_h\})$.

In Lemma \ref{corre2} we shall  prove that $K$ is well-defined
and that  $\widetilde {   M}$ naturally inherits
a semilattice operation $  \vee$ from the semilattice product
$( M, \vee_{_{ \mathbf   M}}) \times \mathbf M^{{<} \omega }$.

For every $n \geq 1$, define $ \sqsubseteq^n $ 
(henceforth denoted simply $ \sqsubseteq $ 
when no risk of ambiguity occurs)
on   $\widetilde {   M}$ by
 \begin{align*} \labbel{df}
 &[a, \{b_1,  \dots, b_h\}] \sqsubseteq^n
[c_1, \{d_{1,1},  \dots, d_{1,k_1}\}], \dots,
[c_n, \{d_{n,1},  \dots, d_{n,k_n}\}]
\text{ if}  
\\ 
& [a, \{b_1,  \dots, b_h\}]  \leq
K[c_1, \{d_{1,1},  \dots, d_{1,k_1}\}] \vee \dots \vee
K[c_n, \{d_{n,1},  \dots, d_{n,k_n}\}],
 \end{align*} 
 where $\leq$ 
is the order induced by $   \vee$.

Let $\widetilde {   \mathbf M} = (\widetilde { M},   \vee, \sqsubseteq ^n)  _{n \geq 1} $,
$\widetilde {   \mathbf M}' = (\widetilde { M},   \vee, K) $.
Finally, define
$\upsilon_{_\mathbf M} :  M \to \widetilde { M}$ 
by 
$\upsilon_{_\mathbf M} (a)= [a,  \emptyset ]$.
\end{definition}

\begin{remark} \labbel{*} 
We  think
of  $[a, \{b_1,  \dots, b_h\}]$ as $a \vee Kb_1 \vee \dots \vee  Kb_h$,
 where $Kb_1, \dots , Kb_h$ 
are the ``new''  closures we need to introduce.
Compare Remark \ref{rough}.
In particular,  $[a, \emptyset ]$ corresponds to $a$
and $[b, \{ b \}]$ corresponds to a  new element $Kb$.  
 \end{remark}   

Now we can state and prove some results similar
to \cite{mttlib,mttna}. The statements are essentially the same,
while the proofs differ in many places.
The fact that the proofs here and in \cite{mttlib}
are quite simpler than the proofs in \cite{mttna}
confirm the idea that 
 specialization semilattices 
(multi-argument specialization semilattices, respectively) are the ``right''
framework for dealing with subreducts of  topological spaces
(closure spaces, respectively).  
On the other hand, while it is still true that
every specialization semilattice is a subreduct of some closure space,
the fact that the universal extension constructed in 
\cite{mttna} is much more involved suggests
that the correspondence between specialization semilattices
and closure spaces is less natural. We
shall provide further arguments in favor of this thesis
in a version of \cite{mtt} under revision.  

\begin{theorem} \labbel{univt2}
Suppose that  $\mathbf M$ is a 
multi-argument specialization semilattice
and let $ \widetilde {  \mathbf M}$
and $\upsilon_{_\mathbf M} $
  be  as in Definition \ref{und2}.
Then the following statements hold.
  \begin{enumerate}   
 \item 
$ \widetilde {  \mathbf M}$ is
a principal regular
multi-argument specialization semilattice.  
\item  
$\upsilon_{_\mathbf M} $ is an
 embedding of $\mathbf M$ 
 into $\widetilde {  \mathbf M}$.
\item  
The pair $ (\widetilde {  \mathbf M}, \upsilon_{_\mathbf M})$
has the following universal property.

For every  principal regular multi-argument 
 specialization semilattice
$\mathbf  T$  
 and every homomorphism
 $ \eta : \mathbf M \to \mathbf  T$, there
is a unique  $K$-{\hspace{0 pt}}homomorphism  
 $ \widetilde{\eta} : \widetilde{\mathbf M} \to \mathbf  T$
such that
$\eta = \upsilon_{_\mathbf M} \circ \widetilde{\eta}$. 
\begin{equation*}
\xymatrix{
	{\mathbf M}  \ar[rd]_{\eta}
 \ar[r]^{\upsilon_{_\mathbf M}}
 &\widetilde{\mathbf  M} \ar@{-->}[d]^{\widetilde{\eta}}\\
	 &{\mathbf  T}}
   \end{equation*}    

\item
Suppose that  $\mathbf U$
is another multi-argument specialization semilattice 
and $\psi : \mathbf M \to \mathbf U$ 
is a 
homomorphism. Then 
 there is a unique
 $K$-\hspace{0 pt}homomorphism
$\widetilde{\psi}: \widetilde{\mathbf M} \to \widetilde{\mathbf U}$
making the following diagram commute:
\begin{equation*}
\xymatrix{
	{\mathbf M} \ar[d]_{\psi} \ar[r]^{\upsilon_{_\mathbf M}}
 &\widetilde{\mathbf M} \ar@{-->}[d]^{\widetilde{\psi}}\\
	{\mathbf U} \ar[r]^{\upsilon_{_\mathbf U}}
 &{\widetilde{\mathbf U}}}
   \end{equation*}    
 \end{enumerate} 
 \end{theorem}

\begin{lemma} \labbel{corre2}
Assume the notation and the definitions in \ref{und2}. 
  \begin{enumerate}[(i)]    
\item
 The relation $\precsim$ from Definition \ref{und2}(a)
is reflexive and transitive on 
$ M \times  M^{{<} \omega }$, hence
 $ \sim$ from \ref{und2}(b) is an equivalence relation. 
\item
The operation $K$ is
 well-defined on the $ \sim$-equivalence classes.
\item
The relation $ \sim$ is a semilattice congruence
on the semilattice  $ ( M, \vee) \times \mathbf M^{{<} \omega }$,
hence the quotient $\widetilde{M}$
inherits  a semilattice structure from
the semilattice $( M, \vee) \times \mathbf M^{{<} \omega }$.
The join operation on the quotient is given by
\begin{equation} \labbel{3-1}
[a, \{b_1,  \dots, b_h\}]  \vee [c, \{d_1,  \dots, d_k\}] =
[a \vee _{_{ \mathbf   M}}  c ,  \{b_1,  \dots, b_h, d_1,  \dots, d_k \}], 
 \end{equation}    
and, moreover, the following holds:
\begin{equation}\labbel{30}     
\begin{aligned} ~     
[a, \{b_1,  \dots, b_h\}] & \leq [c, \{d_1,  \dots, d_k\}] 
\text{ if and only if}  
\\
(a, \{b_1,  \dots, b_h\}) & \precsim (c, \{d_1,  \dots, d_k\}).
 \end{aligned}
 \end{equation}    
 \end{enumerate} 
 \end{lemma} 

\begin{proof}  
(i) To prove that 
$\precsim$ is reflexive, notice that
if $a=c$, then condition (a1) is verified 
by an arbitrary choice of $d$,
say, $d=b_1$ (notice that (a1) is verified by definition 
if $a=c$ and $\{d_1, d_2, \dots, d_k\}$ is empty).
  
If  $ \{b_1, b_2, \dots, b_h\}= 
\{d_1, d_2, \dots, d_k\}$, then 
condition (a2) is verified by \eqref{m1}.

In order to prove transitivity of $\precsim$,  suppose that
\begin{equation*} 
(a, \{b_1, \dots, b_h\}) \precsim 
(c,\{d_1,  \dots, d_k\})\precsim 
(e,\{f_1, \dots, f_\ell\}).
 \end{equation*}     
By (a1), 
 $a \leq_{_{ \mathbf   M}} c
 \vee_{_{ \mathbf   M}} d$, for some
$d$ such that  
$d \sqsubseteq_{_{ \mathbf   M}} d_1, \dots , d_k$, 
and 
$c \leq_{_{ \mathbf   M}} e
 \vee_{_{ \mathbf   M}} f$, for some
$f$ such that 
$f \sqsubseteq_{_{ \mathbf   M}} f_1, \dots , f_ \ell$. 
Thus 
 $a \leq_{_{ \mathbf   M}} e
 \vee_{_{ \mathbf   M}} d \vee_{_{ \mathbf   M}} f$.    
By (a2), for every 
$j \leq k$, there is $m \leq \ell$
such that $d_j \sqsubseteq_{_{ \mathbf   M}} f_m$,
hence 
$d \sqsubseteq_{_{ \mathbf   M}} f_1, \dots , f_ \ell$,
by \eqref{m2+} and, possibly , \eqref{m4},  \eqref{m5}
and \eqref{m6}.
Finally, by \eqref{m7}, we get  
$ d \vee_{_{ \mathbf   M}} f \sqsubseteq_{_{ \mathbf   M}} f_1, \dots , f_ \ell$, thus the element
$d \vee_{_{ \mathbf   M}} f$ witnesses 
 condition $(a1)$ for the relation 
$(a, \{b_1, \dots, b_h\}) \precsim 
(e,\{f_1, \dots, f_\ell\})$.  
Condition (a2) holds
since,  
for every $i \leq h$, 
there is $j \leq k$ such that   
$b_i \sqsubseteq_{_{ \mathbf   M}} d_j$
and, 
for every 
$j \leq k$, there is $m \leq \ell$
such that $d_j \sqsubseteq_{_{ \mathbf   M}} f_m$, hence we get
$b_i \sqsubseteq_{_{ \mathbf   M}} f_m$
by \eqref{m2-}.

Since $\precsim $ is reflexive and transitive, then so is
$\sim$; hence $\sim$ is an equivalence
relation, being symmetric by definition.   

(ii) It is enough to show that
\begin{equation} \labbel{mum} 
\text{if  
$(a, \{b_1,  \dots, b_h\}) \precsim
(c, \{d_1,  \dots, d_k\})$, then
$ (a, \{\bar b\} ) \precsim (c, \{ \bar d  \} )$,}
 \end{equation}   
where we have set
$\bar b= a \vee_{_{ \mathbf   M}} b_1 \vee_{_{ \mathbf   M}}
 \dots \vee_{_{ \mathbf   M}} b_h$ and
$\bar d= c \vee_{_{ \mathbf   M}} d_1 \vee_{_{ \mathbf   M}} 
\dots \vee_{_{ \mathbf   M}} d_k  $. 

If $(a, \{b_1,  \dots, b_h\}) \precsim
(c, \{d_1,  \dots, d_k\})$, then,
by (a1), 
there is
$d \sqsubseteq_{_{ \mathbf   M}} d_1, 
\dots,  \allowbreak  d_k$
such that $a \leq _{_{ \mathbf   M}} c \vee_{_{ \mathbf   M}} d$.
By \eqref{m9}  and \eqref{m2*} we get 
$d \sqsubseteq_{_{ \mathbf   M}}  \bar d$,
hence clause (a1) is satisfied witnessing
$ (a, \{\bar b\} ) \precsim (c, \{ \bar d  \} )$.

Since $ c \sqsubseteq \bar d$, by \eqref{m1+}, we get
$a \sqsubseteq  _{_{ \mathbf   M}} c \vee_{_{ \mathbf   M}} d
\sqsubseteq  _{_{ \mathbf   M}} \bar d$, by \eqref{m1+} again
and \eqref{m7} using  the already proved 
$d \sqsubseteq_{_{ \mathbf   M}}  \bar d$.
Hence $a \sqsubseteq  _{_{ \mathbf   M}}  \bar d$, by \eqref{m2-}.
Since, for every $i \leq h$, 
there is $j \leq k$ such that   
$b_i \sqsubseteq_{_{ \mathbf   M}} d_j$, by (a2),
then,
 for every $i \leq h$, $b_i \sqsubseteq_{_{ \mathbf   M}} \bar d$
by \eqref{m2*}. Iterating \eqref{m7} we then get
  $ \bar b\sqsubseteq_{_{ \mathbf   M}} \bar d$,
and this proves condition (a2), thus the proof of 
(ii) is complete. 

(iii) By symmetry, it is enough to show that
if  
\begin{equation}\labbel{uffa} 
   (a, \{b_1,  \dots, b_h\}) \precsim
(c, \{d_1,  \dots, d_k\}),
  \end{equation} 
  then   
\begin{equation*} 
(a, \{b_1,  \dots, b_h\})  \vee (e, \{f_1,  \dots, f_\ell\}) \precsim
(c, \{d_1,  \dots, d_k\}) \vee (e, \{f_1,  \dots, f_\ell\}),
 \end{equation*}     
  that is, 
\begin{equation}\labbel{uffi}      
(a \vee_{_{ \mathbf   M}} e, \{b_1,  \dots, b_h, f_1,  \dots, f_\ell\}) ) \precsim
(c \vee_{_{ \mathbf   M}} e, \{d_1,  \dots, d_k, f_1,  \dots, f_\ell\}).
 \end{equation}
Equation \eqref{uffa}
is witnessed by  $a \leq_{_{ \mathbf   M}} c
 \vee_{_{ \mathbf   M}} d$, for some
$d \sqsubseteq_{_{ \mathbf   M}} d_1, 
\dots,  \allowbreak  d_k$.
By iterating \eqref{m6}, we have 
$d \sqsubseteq_{_{ \mathbf   M}} d_1, 
\dots,  \allowbreak  d_k, f_1,  \dots, f_\ell$
and this gives (a1) for \eqref{uffi}.
Clause (a2) for \eqref{uffi} follows from
clause (a2) for \eqref{uffa} and \eqref{m1}.

Condition \eqref{30},
is proved exactly as the corresponding condition
(3.2) in
\cite[Lemma 3.6]{mttna}. 
\end{proof}

\begin{proof}[Proof of Theorem \ref{univt2}]
Definition \ref{und2} is justified
by Lemma \ref{corre2}.

\begin{claim} \labbel{cl2} 
$\widetilde {   \mathbf M}' = (\widetilde { M},   \vee, K) $
is  a closure semilattice.
 \end{claim}  

By Lemma \ref{corre2}(iii),   
$(\widetilde { M},   \vee)$ is a semilattice; the fact
that $K$ is a closure operation
 follows from \eqref{30} and \eqref{mum}.
See \cite[Claim 3.7]{mttna} for details. 

 Clause (1)
in Theorem \ref{univt2} follows 
from Claim \ref{cl2} and Proposition \ref{vere}.

Clause (2) is proved exactly as the corresponding
clause in \cite[Theorem 3.5]{mttna}.

Some new arguments are needed in order to  deal with (3).
If $\eta: \mathbf M \to \mathbf T$
is a homomorphism and there exists 
 $\widetilde{\eta}$ 
such that   
$\eta = \upsilon_{_\mathbf M} \circ \widetilde{\eta}$,
then necessarily 
$\widetilde{\eta}([a, \emptyset ]) =
   \widetilde{\eta}(\upsilon_{_\mathbf M}(a))
 = \eta (a) $, for every $a \in M$.
If furthermore $ \widetilde{\eta}$
is a $K$-homomorphism, then 
$ \widetilde{\eta}([b, \{ b \}]) =  K_{_\mathbf T} \eta (b)$, hence
necessarily 
\begin{equation}\labbel{nec} 
    \widetilde{\eta}([a, \{b_1, \dots, b_h   \} ]) \allowbreak =
 \eta (a) \vee_{_\mathbf T} K_{_\mathbf T} \eta (b_1) \vee_{_\mathbf T}
 \dots \vee_{_\mathbf T} K_{_\mathbf T} \eta (b_h).  \end{equation}
See \cite{mttna} for full details.
In particular, if $   \widetilde{\eta}$  
exists it is unique.

It remains  to show that the above condition
\eqref{nec} 
actually determines a $K$-homomorphism
$   \widetilde{\eta}$ from 
$\widetilde{ \mathbf M}$  
to $\mathbf T$. 
First, we need to check that 
if  
$(a, \{b_1,  \dots, b_h\}) \sim
(c, \{d_1,  \dots, d_k\})$, then 
$\eta (a) \vee_{_\mathbf T} K_{_\mathbf T} \eta (b_1) \vee_{_\mathbf T} \dots
 \vee_{_\mathbf T} K_{_\mathbf T} \eta (b_h)
= \eta (c) \vee_{_\mathbf T} K_{_\mathbf T} \eta (d_1) \vee_{_\mathbf T}
 \dots \vee_{_\mathbf T} K_{_\mathbf T} \eta (d_k)$, so that
$   \widetilde{\eta}$ is well-defined. 

Assume that clauses (a1) and (a2) in Definition \ref{und2} hold.
From $d \sqsubseteq_{_{ \mathbf   M}} d_1,
 \dots , d_k$ 
we get $ \eta (d) \sqsubseteq_{_\mathbf T} \eta (d_1), \dots, \eta (d_k)$,
since $\eta$ is a homomorphism,
hence
$\eta (d) \leq_{_\mathbf T} K_{_\mathbf T}\eta (d_1) 
\vee_{_\mathbf T} \dots \vee_{_\mathbf T} K_{_\mathbf T}\eta (d_k)$,
since $\mathbf T$ is assumed to be regular.  
From $a \leq_{_\mathbf M} c \vee_{_\mathbf M} d$, 
given by (a1), we get
$ \eta (a)  \leq_{_\mathbf T} \eta (c) \vee_{_\mathbf T} \eta (d)
\leq_{_\mathbf T} 
\eta (c) \vee_{_\mathbf T} 
 K_{_\mathbf T}\eta (d_1) 
\vee_{_\mathbf T} \dots \vee_{_\mathbf T} K_{_\mathbf T}\eta (d_k)$.
By \ref{und2}(a2),
for every $i \leq h$, there is 
 $j \leq k$ such that  
  $b_i \sqsubseteq_{_\mathbf M} d_j$. 
By Remark \ref{psc}(b), we get  
$ K_{_\mathbf T}\eta (b_i) \leq_{_\mathbf T} K_{_\mathbf T}\eta (d_j)$,
hence 
\begin{align*}  
\eta (a) \vee_{_\mathbf T} K_{_\mathbf T} \eta (b_1) \vee_{_\mathbf T}
 \dots \vee_{_\mathbf T} K_{_\mathbf T} \eta (b_h)
\leq_{_\mathbf T}
\\
 \eta (c) \vee_{_\mathbf T} K_{_\mathbf T}
 \eta (d_1) \vee_{_\mathbf T} \dots \vee_{_\mathbf T} K_{_\mathbf T} \eta (d_k).
\end{align*} 
The converse inequality is proved in the symmetrical way.
We have proved  that $   \widetilde{\eta}$ is well-defined.

That $   \widetilde{\eta}$ is a
semilattice homomorphism and a $K$-homomorphism
follow from \eqref{3-1},  
the definitions of $K$ and $\widetilde{\eta}$,
the assumption that $\eta$ is a homomorphism 
and 
the fact that $K_{_\mathbf T}$ is a closure operation.
See \cite{mttna} for more details.

In the present situation we need also  check that
 $   \widetilde{\eta}$ is a homomorphism of 
multi-argument specialization semilattices, 
namely, that $   \widetilde{\eta}$ preserves 
$ \sqsubseteq ^n$, for every $n$.
 However, this is immediate from Claim \ref{cl2},
the definition  of $ \sqsubseteq $ in
Definition \ref{und2} 
and Proposition \ref{vere}. 

Clause (4) follows from Clause (3)
applied to   $\eta = \psi \circ \upsilon_{_{\mathbf U} }$.
 \end{proof}

It is necessary to ask that 
 $\widetilde{\eta}$ is 
a $K$-homomorphism in Theorem \ref{univt2}(3);
compare a parallel observation shortly before 
Remark 3.4 in \cite{mttlib}.
On the other hand, as already remarked, 
 in Theorem \ref{univt2}(3) $\eta$ is not required to
preserve existing closures.

\section{Embedding into a closure space} \labbel{embsec} 

In the previous section
we have showed that every multi-argument specialization semilattice 
can be embedded into the multi-argument  specialization reduct
of a closure semilattice. Some arguments form \cite{mtt}
imply that every closure semilattice can be embedded into a closure space,
hence we get the corresponding result for  
multi-argument specialization semilattices.
To be strictly formal, in the next theorem a closure space
is considered as a structure of the form
$(\mathcal P(X), \cup, K)$. 

\begin{theorem} \labbel{emb}
Every closure semilattice can be embedded into
some closure space.
 \end{theorem} 

\begin{proof}
Suppose that $\mathbf S = (S, \vee_{_\mathbf S}, K_{_\mathbf S} )$
is a closure semilattice and 
let $\varphi: S \to \mathcal P(S)$
be the function defined by  
$ \varphi (a) = \nup a = 
\{ b \in S \mid a \nleq_{_\mathbf S} b \}  $. 
We shall give  $\mathcal P(S)$ the structure of 
a closure space in such a way that 
$\varphi$  becomes an embedding of closure semilattices.

Clearly, $\varphi$  is a semilattice embedding, no matter the closure
operation.
Let $ \mathfrak F = \{ \, C \subseteq S \mid
 C= \varphi ( z), \text{ for some
$z \in S $ such that } K_{_\mathbf S}z=z \,\} $.
For $x \subseteq S$, set $Kx= \bigcap \{ \, C \in  \mathfrak F  \mid  x \subseteq C\, \} $,
where, conventionally, $S$ is the intersection of the empty family.
Clearly, $K$ is a closure operation, hence it remains to show that
$\varphi$  is a $K$-homomorphism.

If $a \in S$,
then $\varphi(K_{_\mathbf S}a) \in \mathfrak F$, 
 since $K_{_\mathbf S}K_{_\mathbf S}a=K_{_\mathbf S}a$,
thus 
$ K \varphi (a) \subseteq \varphi(K_{_\mathbf S} a)$, since
$\varphi (a) \subseteq \varphi(K_{_\mathbf S} a)$.

We now show the converse. If $\varphi (a) \subseteq C$
and $C \in \mathfrak F$, 
then
$C= \varphi ( z)$, for some
$z \in S $ such that $ K_{_\mathbf S}z=z $.
Thus $\varphi ( a) \subseteq \varphi (z)$,
hence $a \leq_{_\mathbf S} z$,
 since $\varphi$  is a semilattice embedding,
in particular, an order embedding. 
We then get $a \leq_{_\mathbf S} z = K_{_\mathbf S}z$,
hence
$ K_{_\mathbf S} a \leq_{_\mathbf S}  K_{_\mathbf S}z$,
thus 
$ \varphi ( K_{_\mathbf S} a)  \leq 
\varphi ( K_{_\mathbf S}z) = C$.
Since the above argument applies to every $C$ 
such that 
$\varphi (a) \subseteq C$, and 
$K \varphi (a)$ is the intersection of all such $C$s,
then 
$\varphi ( K_{_\mathbf S} a)  \leq K \varphi (a)$.

In conclusion,  $\varphi ( K_{_\mathbf S} a)  = K \varphi (a)$
and the proof is complete.
\end{proof}

\begin{corollary} \labbel{corm} 
Every multi-argument specialization semilattice
can be embedded into the multi-argument specialization semilattice
associated to some closure space.
\end{corollary}
  
 \begin{proof}
Immediate from Theorems \ref{univt2}(1)(2)
and \ref{emb}, Proposition \ref{vere}(2) and 
Remark \ref{psc}(a). 
 \end{proof}


\begin{thebibliography}{99}    



\bibitem{E} Ern\'{e}, M.,
 \emph{Closure},
in Mynard, F.,  Pearl E. (eds), 
``Beyond topology'',
  Contemp. Math.
\textbf{486},
Amer. Math. Soc., Providence, RI,
163--238
(2009).



\bibitem{H} 
 Hodges, W.,
\emph{Model theory},
Encyclopedia of Mathematics and its Applications
\textbf{42},
Cambridge University Press, Cambridge
(1993).


\bibitem{mtt} Lipparini, P., \emph{A model theory of topology},
arXiv:2201.00335v1, 1--30 (2022).

\bibitem{mttlib}  Lipparini, P., 
\emph{Universal extensions of specialization semilattices}, 
Categ. Gen. Algebr. Struct. Appl. \textbf{17}, 
101--116  (2022).

\bibitem{mttna}
  Lipparini, P., 
\emph{Universal specialization semilattices}, 
arXiv:2207.11745v2, 1--15 (2022).

\bibitem{MT}  McKinsey, J. C. C., Tarski, A., \emph{The algebra of
   topology}, Ann. of Math.  \textbf{45}, 141--191 (1944).


\end{thebibliography}
\end{document}